\documentclass[a4paper,oneside,12pt,reqno]{amsart}
\usepackage{geometry}

\usepackage[svgnames]{xcolor}
\usepackage[colorlinks=true]{hyperref}

\usepackage{amssymb,amsmath,amsthm,latexsym}
\usepackage[mathscr]{eucal}
\usepackage{mathtools}
\usepackage{mathrsfs}

\newtheorem{thm}{Theorem}
\newtheorem{lem}[thm]{Lemma}

\theoremstyle{definition}

\title{Ramsey numbers of connected 4-clique matchings}
\author{Krit Kanopthamakun}
\address{Division of Computational Science, Faculty of Science, Prince of Songkla University, Thailand}
\email{kritmpsu@hotmail.com}

\author{Panupong Vichitkunakorn}
\address{Division of Computational Science, Faculty of Science, Prince of Songkla University, Thailand}
\email{panupong.v@psu.ac.th}

\begin{document}
\maketitle

\begin{abstract}
In this paper, we determine the exact value of the $2$-edge-coloring Ramsey number of a connected $4$-clique matching $c(nK_4)$, which is a set of connected graphs containing an $nK_4$ is $13n-3$ for any positive integer $n \geq 3$.
This is an extension of the result by Roberts (2017), which is proved only for $n\geq 18$.
We also show that the result still holds when $n=2$ provided that $R_2(2K_4) \leq 23$.
\end{abstract}

\section{Introduction and Preliminaries}

Given graphs $G_1,G_2,\ldots,G_k$, the \emph{(graph) Ramsey number} $R(G_1,\ldots,G_k)$ is the smallest integer $n$ such that every $k$-edge-coloring of $K_n$ contains a copy of $G_i$ in color $i$ for some $i$. When $G_i = G$ for all $i$, we write $R_k(G)$ for $R(G_1,\ldots,G_k)$. When $G_i$ is a complete graph $K_{n_i}$ for all $i$, we abbreviate the notation to $R(n_1,\ldots,n_k)$.

Until now, no exact formula is known for general values of $n_i$ and few $2$-color Ramsey numbers have been computed: $R(3,3)=6, R(3,4)=9, R(3,5)=14, R(3,6)=18, R(3,7)=23, R(3,8)=28, R(3,9)=36, R(4,4)=18$ and $R(4,5)=25$. See the survey by Radziszowski \cite{radz} and references therein. 

There is a lot of research on Ramsey numbers for non-complete graphs such as complete bipartite graphs, cycles, paths, stars, etc \cite{radz}.
For disconnected graphs, there are only few results. For example, the exact formula of $R(n_1P_2,n_2P_2,\ldots,n_kP_2)$ was shown in \cite{coclo}, and was extended \cite{lorp} to $R(K_m,n_1P_2,n_2P_2,\ldots,n_kP_2)$. Later, there are more results on general cases of disconnected graphs, such as multiple copies of complete graphs \cite{lorse,burrer,coclo,lorra}.
For $R_2(nK_r)$, the exact formula for any $n \geq 2$ is only known when $r$ is either $2$ \cite{coclo,lorp} or $3$ \cite{burrer}. For $r = 4$, \cite{burr} gave the formula when $n$ is sufficiency large.

Let $G$ be a disconnected graph and $c(G)$ be the set of all connected graphs containing $G$. Then $R_k(c(G))$ is the least integer $n$ such that any $k$-edge-coloring on $K_n$ contains a monochromatic copy of a graph in $c(G)$, see \cite{gyasa,rob}.
Cockayne and Lorimer first gave a result for a $2$-color connected matching, $R_2(c(nK_2))$, \cite{coclo}.
Gy{\'a}rf{\'a}s and S{\'a}rk{\"o}zy determined the exact value for a connected triangle matching \cite{gyasa}.
Later Roberts gave a general result for a connected clique matching \cite{rob}, which is shown the following theorem.

\begin{thm} \label{ram} \textup{\cite{rob}} 
	For $r \geq 4$ and $n \geq R(r,r)$ we have
	\[R_2(c(nK_r)) = (r^2-r+1)n-r+1.\]
\end{thm}

The lower bound of Theorem \ref{ram} can be established using a result by Burr \cite{burr2} or a constructive proof given by Roberts \cite{rob}, both of which hold for all $r \geq 3$ and $n \geq 2$.

In the proof of Theorem \ref{ram}, the condition of $n \geq R(r,r)$ is necessary as it ensures that any $2$-edge-coloring on the complete graph of $(r^2-r+1)n-r+1$ vertices has enough vertices to form at least $r-1$ blue connected components.
However, when $n < R(r,r)$, the number of blue connected components can be fewer than $r-1$.
To address this issue, we introduce a new method for handling smaller values of $n$.

In this paper, we are only interested in the case when $r=4$.
Theorem \ref{ram} showed that $R_2(c(nK_4)) = 13n-3$ when $n \geq R(4,4) = 18$.
The main goal of this paper is to extend the result of Theorem \ref{ram} from $n \geq 18$ to $n \geq 3$. 
We also show that the formula still holds when $n=2$ if $R_2(2K_4) \leq 23$.

\section{Main results}\label{sec:main}

Now that we already obtained the lower bound, we turn our attention to proving the upper bound.
The next lemma is a main tool for proving Theorem~\ref{main}.

\begin{lem} \label{lem:2part}
	Let $n \geq 2$ and $k \geq 2$ be positive integers with $k \geq 2n-3 + \left \lfloor{\frac{n-1}{4}}\right \rfloor$. 
	If $k = a_1+a_2+\dots+a_t$ where $n-1 \geq a_1 \geq a_2 \geq \dots \geq a_t \geq 1$, 
	then $a_1,a_2,\ldots,a_t$ can be combined into three parts with each part being at least $\left \lceil{\frac{n}{4}}\right \rceil$ or two parts with each part being at least $n-2$.
\end{lem}

\begin{proof}
	We first consider two cases for $a_1$.
	
	\noindent \textbf{Case 1: $a_1 \leq n-3$.} 
	We consider three subcases.
	
	\textbf{Case 1.1:} $a_1,a_2 \geq \left \lceil{\frac{n}{4}} \right \rceil$.
	Then $a_3+\cdots+a_t \geq 3+\left \lfloor{\frac{n-1}{4}}\right \rfloor > \left \lceil{\frac{n}{4}} \right \rceil$.
	Thus, we can combine $a_3,a_4,\ldots,a_t$ into a single part.
	Hence, we have three parts where each part is at least $\left \lceil {\frac{n}{4}} \right \rceil$.
	
	\textbf{Case 1.2:} $a_1 \geq \left \lceil{\frac{n}{4}} \right \rceil$ but $a_2 < \left \lceil{\frac{n}{4}} \right \rceil$.
	Then $a_2+\cdots+a_t \geq n+\left \lfloor{\frac{n-1}{4}}\right \rfloor > n-1$.
	Since $a_i < \left \lceil{\frac{n}{4}} \right \rceil$ for $2 \leq i \leq t$, those $a_2,\ldots,a_t$ can be combined into two parts where each part is at least $\left \lfloor{\frac{n-1}{4}}\right \rfloor+1 = \left \lceil{\frac{n}{4}}\right \rceil$.
	So, we obtain three parts where each part is at least $\left \lceil{\frac{n}{4}} \right \rceil$.
	
	\textbf{Case 1.3:} $a_1 < \left \lceil{\frac{n}{4}} \right \rceil$.
	Then $n \geq 5$. We will combine $a_1,a_2,\ldots,a_t$ into three parts, says $x,y$ and $z$,
	such that the difference between any two parts is less than $\left \lceil{\frac{n}{4}} \right \rceil$.
	This is possible because $a_i < \left \lceil{\frac{n}{4}} \right \rceil$.
	Next, we suppose that $x \geq y \geq z$. 
	We have
	\begin{equation*}
		3z + 2\left(\left \lceil{\frac{n}{4}} \right \rceil-1\right) \geq x+y+z \geq 2n-3+\left \lfloor{\frac{n-1}{4}}\right \rfloor.
	\end{equation*}
	Then we obtain
	\begin{equation*}
		z \geq \frac{2n-1}{3}+\frac{1}{3}\left \lfloor{\frac{n-1}{4}}\right \rfloor-\frac{2}{3} \left \lceil{\frac{n}{4}} \right \rceil \geq \left \lceil{\frac{n}{4}} \right \rceil.
	\end{equation*}
	Again, we get three parts where each part is at least $\left \lceil{\frac{n}{4}} \right \rceil$.
	
	\noindent \textbf{Case 2: $a_1 \geq n-2$.}     
	Again, we will split it into three subcases.
	
	\textbf{Case 2.1:} $a_2,a_3 \geq \left \lceil{\frac{n}{4}} \right \rceil$.
	Then we get three parts of size at least $\left \lceil{\frac{n}{4}} \right \rceil$.
	
	\textbf{Case 2.2:} $a_2 \geq \left \lceil{\frac{n}{4}} \right \rceil$ but $a_3 < \left \lceil{\frac{n}{4}} \right \rceil$.
	We consider the value of $a_3+a_4+\cdots+a_t$. If it is at least $\left \lceil{\frac{n}{4}} \right \rceil$, then combining $a_3,a_4,\ldots,a_t$ into a single part.  
	So, we obtain three parts of size at least $\left \lceil{\frac{n}{4} } \right \rceil$.
	Otherwise, $a_1+a_2 > 2n-3+\left \lfloor{\frac{n-1}{4}}\right \rfloor-\left \lceil{\frac{n}{4}} \right \rceil=2n-4$.
	Then $a_2 \geq n-2$.
	Combining $a_2,a_3,\ldots,a_t$ into a single part, we yield two parts of size at least $n-2$.
	
	\textbf{Case 2.3:} $a_2 < \left \lceil{\frac{n}{4}} \right \rceil$.
	Then $n \geq 5$ and $a_2+a_3+\cdots+a_t \geq n-2+\left \lfloor{\frac{n-1}{4}}\right \rfloor$. 
	When $5 \leq n \leq 8$, then $k \geq 8$ and $a_2+a_3+\cdots+a_t \geq n-1$.
	But, $a_i < \left \lceil{\frac{n}{4}} \right \rceil = 2$ for all $2 \leq i \leq t$.
	Thus, $5 \leq n \leq t$.
	Therefore, we can combine $a_2,a_3,\ldots,a_t$ into two parts, where each part is at least $2 = \left \lceil{\frac{n}{4}} \right \rceil$.
	Next, suppose that $n \geq 9$.
	We have $a_2+a_3+\cdots+a_t > n-1$.
	Since $a_i < \left \lceil{\frac{n}{4}} \right \rceil$ for $2 \leq i \leq t$, those $a_2,a_3,\ldots,a_t$ can be combined into two parts where each part is at least $\left \lceil{\frac{n}{4}} \right \rceil$.
	Therefore, we have three parts where each part is at least $\left \lceil{\frac{n}{4}} \right \rceil$.
	
	From all cases, we can combine $a_1,a_2,\ldots,a_t$ into three parts with each part being at least $\left \lceil{\frac{n}{4}} \right \rceil$ or into two parts with each part being at least $n-2$.
\end{proof}

Note that, by Lemma \ref{lem:2part}, we can combine $a_1,a_2,\ldots,a_t$ into three parts where each part is at least $\left \lceil{\frac{n}{4}}\right \rceil$, except when $a_1,a_2 \geq n-2$ and $a_3+a_4+\cdots+a_t < \left \lceil{\frac{n}{4}}\right \rceil$.

The following theorems are additional tools for proving the main theorem.

\begin{thm} \label{thm:mK2} \textup{\cite{fau}}
	Let $G$ be a $k$-regular graph with $n$ vertices, where $1 \leq k \leq 6$ and $m \geq 1$. Then
	\begin{equation*}
		R(mK_2,G)=\max\{n+2m-\alpha(G)-1,n+m-1\},
	\end{equation*}
	where $\alpha(G)$ is an independent number of $G$.
\end{thm}

In particular, when $G=(k+1)K_4$, where $k \geq 1$,
it can be explicitly written as
$$R\left(mK_2, (k+1)K_4\right) = \begin{cases} 3k+2m+2, & k< m-1, \\ 4k+m+3, & k\geq m-1,\end{cases}$$
where both cases are equal when $k=m-1$.

We find it convenient to use the following simplified alternative to Theorem \ref{thm:mK2}.

\begin{lem} \label{lem:fu}
	Let $G$ be a $2$-edge-coloring on $K_{4k+u}$ with red and blue where $k \geq 1$ and $u \geq 0$.
	Suppose that $G$ contains at most $k$ disjoint blue $K_4$.
	Then there exists at least $f_k(u)$ red copies of $K_2$ where
	$$f_k(u) = \begin{cases}
		0, & u \leq 3, \\
		u-3, & 4 \leq u \leq k+3, \\
		\left \lfloor{\frac{u+k-2}{2}}\right \rfloor, & u \geq k+4.
	\end{cases}$$ 
\end{lem}

\begin{proof}
	The first case when $u \leq 3$ is obvious.
	For $4 \leq u \leq k+3$. Consider $m = u-3$. Then $m \leq k$.
	By Theorem \ref{thm:mK2}, $R(mK_2,(k+1)K_4)=4k+m+3=4k+u=n(G)$.
	If $u \geq k+4$, then consider $m =  \left \lfloor{\frac{u+k-2}{2}}\right \rfloor$.
	Thus $m \geq \frac{u+k-3}{2} \geq \frac{2k+1}{2} = k+\frac{1}{2}$.
	Since $m \in \mathbb{Z}$, $k \leq m-1$.
	Again by Theorem \ref{thm:mK2}, $R(mK_2,(k+1)K_4)=3k+2m+2 \leq 3k+2\left(\frac{u+k-2}{2}\right)+2 = 4k+u = n(G)$. 
	From all cases, $G$ contains a red matching of size $m = f_k(u)$ as desired.
\end{proof}

\begin{thm} \label{lem:mK_3} \textup{\cite{lorra}}
	For $m\geq 1$ and $n\geq 2$, we have 
	\begin{equation*}
		R(mK_3,nK_4) = \max\{3n+3m+1,4n+2m+1\}.
	\end{equation*}
\end{thm}

Theorem \ref{lem:mK_3} can be explicitly written for positive integers $m$ and $k$ as 
$$R\left(mK_3, (k+1)K_4\right) = \begin{cases} 3k+3m+4, & k< m-1, \\ 4k+2m+5, & k\geq m-1.\end{cases}$$

Now that we have all the necessary tools we need,
we will proceed to prove the main result.

\begin{thm} \label{main}
	$R_2(c(nK_4)) = 13n-3$ for $n \geq 3$.
	Furthermore, if $R_2(2K_4) \leq 23$ then $R_2(c(2K_4)) = 23$.
\end{thm}

\begin{proof}
	We consider $G$ a $2$-edge-coloring of $K_{13n-3}$ with red and blue. 
	We will show that there is a monochromatic $c(nK_4)$ in $G$.
	Since a graph and its complement cannot be both disconnected, at least one color class of $G$ is connected, we suppose that the red is connected. 
	
	If $G$ contains a red $nK_4$ or a blue $c(nK_4)$, then we are done. 
	So, we may assume that there are at most $n-1$ disjoint red copies of $K_4$ and no blue $c(nK_4)$. 
	
	Since $13n-3 = 18+4(3n-5)+n-1 = R(4,4)+4(3n-5)+n-1$, the graph $G$ has at least $3n-4+\left \lfloor{\frac{n-1}{4}}\right \rfloor$ disjoint monochromatic copies of $K_4$. 
	Since there are at most $n-1$ disjoint red copies of $K_4$, there are at least $3n-4+\left \lfloor{\frac{n-1}{4}}\right \rfloor-(n-1) = 2n-3 + \left \lfloor{\frac{n-1}{4}}\right \rfloor$ disjoint blue copies of $K_4$.
	
	Let $k$ be the maximum number of disjoint blue copies of $K_4$ in $G$, and $K$ be the set of such $k$ disjoint blue copies of $K_4$.
	We note from the above discussion that $k\geq 2n-3 + \left \lfloor{\frac{n-1}{4}}\right \rfloor$.
	Suppose that blue connected components (including singletons) have vertex sets $V_1,V_2,\ldots,V_t,\ldots,V_{t+s}$.
	We assume that, for $1 \leq i \leq t$, the graph $G[V_i]$ contains $a_i > 0$ blue copies of $K_4$ in $K$, while $G[V_{t+1}],\ldots,G[V_{t+s}]$ contains no blue $K_4$.
	Let $V'=V_{t+1} \cup \cdots \cup V_{t+s}$.
	Since there is no blue $c(nK_4)$ in $G$, we have $a_i\leq n-1$.
	
	Next, we show that $t\geq 2$. Suppose $t=1$, we have
	$$ n-1\geq a_1 = k \geq 2n-3 + \left \lfloor{\frac{n-1}{4}}\right \rfloor. $$
	If $n>2$, the inequality becomes $n-1 \geq 2n-3 + \left \lfloor{\frac{n-1}{4}}\right \rfloor > n-1$, which is a contradiction.
	If $n=2$, then $k=1$. We have $G$ has only one disjoint blue $K_4$.
	As we assume that $R_2(2K_4) \leq 23$, we can conclude from $n(G) = 23 \geq R_2(2K_4)$ that $G$ must contain a red $2K_4$, which is a contradiction.
	This is the only case that $R_2(2K_4) \leq 23$ is required.
	From now on, we will prove this theorem by omitting the case when $n=2$ and $k=1$.
	
	By Lemma \ref{lem:2part}, $a_1,a_2,\ldots,a_t$ can be combined into three parts where each part is at least $\left \lceil{\frac{n}{4}}\right \rceil$ or two parts where each part is at least $n-2$.
	We will show a contradiction that there exists a red $nK_4$ in $G$.
	
	\noindent \textbf{Case 1:} $a_1,a_2,\ldots,a_t$ can be combined into three parts where each part is at least $\left \lceil{\frac{n}{4}}\right \rceil$
	
	Since $G[V_i]$ contains a blue $a_iK_4$, we can combine $V_1,V_2,\ldots,V_t$, and $V'$ into three sets, says $U_1,U_2$ and $U_3$, where $G[U_i]$ has $k_i \geq \left \lceil{\frac{n}{4}}\right \rceil$ blue copies of $K_4$ from $K$ and $u_i$ be the number of vertices in $U_i$ that do not belong to any of blue $K_4$ in $K$. We will split the proof into two subcases.
	
	\textbf{Case 1.1:} $k_i \leq n-1$ for all $1 \leq i \leq 3$.    
	Recall that the edges between $U_1, U_2$ and $U_3$ are all red.
	In this subcase, we will form a red $K_4=K_2\vee K_1 \vee K_1$ in $G$ by joining a red $K_2$ (from one component) and two copies of $K_1$ (one from each of the remaining two components). 
	To obtain a red $nK_4$ in $G$, we require the total number of red copies of $K_2$ from each component to be at least $n$.
	By Lemma \ref{lem:fu}, $G[U_i]$ contains at least $f_{k_i}(u_i)$ red copies of $K_2$.
	First, we have to show that
	\begin{equation*} 
		f_{k_1}(u_1)+f_{k_2}(u_2)+f_{k_3}(u_3) \geq n.
	\end{equation*}
	We can assume that $k_1 \geq k_2 \geq k_3$.
	Since $u_1+u_2+u_3= 13n-3-4k$ and $f_{k_1}(u_1)+f_{k_2}(u_2)+f_{k_3}(u_3) \geq f_{k_1}(3)+f_{k_2}(3)+f_{k_3}(13n-9-4k)$,
	the total number of disjoint red copies of $K_2$ in $G[U_1]\cup G[U_2] \cup G[U_3]$ is at least $f_{k_3}(13n-9-4k)$.
	Since $k_i \leq n-1$ and $k=k_1+k_2+k_3$, we have $$13n-9-4k \geq 13n-9-4(3n-3) = n+3 \geq  k_3+4.$$
	By Lemma \ref{lem:fu}, we have 
	\[ f_{k_3}(13n-9-4k) = \left \lfloor{\frac{13n-11-4k+k_3}{2}}\right \rfloor. \]
	Since $k_i \leq n-1$ and $k = k_1+k_2+k_3$, we have
	\[13n-11-4k+k_3 = 13n-11-4k_1-4k_2-3k_3 \geq 13n-11-11(n-1) \geq 2n.\]
	Hence, $f_{k_1}(u_1)+f_{k_2}(u_2)+f_{k_3}(u_3) \geq f_{k_3}(13n-9-4k) \geq n.$
	
	We pick $m_i \leq f_{k_i}(u_i)$ red copies of $K_2$ in $G[U_i]$ so that $m_1+m_2+m_3 = n$.
	Therefore, we have 
	\begin{equation*}
		|U_i| = 4k_i+u_i \geq 4k_i+m_i \geq n+m_i = (m_1+m_2+m_3)+m_i.
	\end{equation*}
	So there exists red $m_iK_2 + \sum_{j\neq i} m_j K_1$ in $G[U_i]$ for $1\leq i \leq 3$.
	Hence, we can construct a red $nK_4$ in $G$.
	
	\textbf{Case 1.2:} $k_i > n-1$ for some $1 \leq i \leq 3$.
	First, for any $1 \leq i \leq 3$, we will partition $U_i$ into two sets say $U_i'$ and $U_i''$ (possibly empty), where the number of blue $K_4$ from $K$ in $G[U_i']$ is from $\left \lceil{\frac{n}{4}}\right \rceil$ to $n-1$.
	We let $U_4' = U_1'' \cup U_2'' \cup U_3''$.
	Note that $U_4' \neq \emptyset$ as $k_i > n-1$ for some $i$.
	Then $U_1',U_2',U_3'$ and $U_4'$ form a new partition of $V(G)$. For $1 \leq i \leq 4$, we let $k_i'$ be the number of disjoint blue $K_4$ from $K$ and $u_i'$ be the number of vertices in $U_i'$ that do not belong to any of blue $K_4$ in $K$.
	We have $\left \lceil{\frac{n}{4}}\right \rceil \leq k_1',k_2',k_3' \leq n-1$.
	
	Note that $|U_i'| \geq n$ for $1 \leq i \leq 3$.
	If $|U_4'| \geq n$, then we can construct a red $nK_4$ by joining one vertex from each component.
	From now on, we can assume that $|U_4'| < n$.
	
	From Lemma \ref{lem:fu}, each component $G[U_i']$, where $1 \leq i \leq 3$, contains at least $f_{k_i'}(u_i')$ red copies of $K_2$. First, we want to show that 
	\begin{equation*} 
		f_{k_1'}(u_1')+f_{k_2'}(u_2')+f_{k_3'}(u_3') \geq n - |U_4'|.
	\end{equation*}
	Since $u_1'+u_2'+u_3'+u_4' = 13n-3-4k$, we get $u_1'+u_2'+u_3'= 13n-3-4k-u_4'$.
	Let $n' = 13n-3-4k-u_4'=13n-3-4(k_1'+k_2'+k_3')-|U_4'|$. From $k_1',k_2',k_3' \leq n-1$, we have $n' \geq n-|U_4'|+9$. 
	We can assume that $k_1' \geq k_2' \geq k_3'$.
	Since $f_{k_1'}(u_1')+f_{k_2'}(u_2')+f_{k_3'}(u_3') \geq f_{k_1'}(3)+f_{k_2'}(3)+f_{k_3'}(n'-6) = f_{k_3'}(n'-6)$, the total number of disjoint red copies of $K_2$ in $G[U_1']\cup G[U_2'] \cup G[U_3']$ is at least $f_{k_3'}(n'-6)$.
	
	By Lemma \ref{lem:fu},
	if $n'-6 \leq k_3'+3$, then $f_{k_3'}(n'-6) = n'-9 \geq n - |U_4'|$.
	If $n'-6 \geq k_3'+4$, then we have
	\[ f_{k_3'}(n'-6) = \left \lfloor{\frac{n'+k_3'-8}{2}}\right \rfloor \geq \frac{n'+k_3'-9}{2} = \frac{13n-4k+k_3'-u_4'-12}{2}.  \]
	Since $n-1 \geq k_1'\geq k_2'\geq k_3', k=k_1'+k_2'+k_3'+k_4'$ and $|U_4'| = 4k_4'+u_4'$, we have
	\begin{align*}
		13n-4k+k_3'-u_4'-12 
		&= 13n-4k_1'-4k_2'-3k_3'-4k_4'-u_4'-12 \\
		&\geq 13n-11(n-1)-4k_4'-u_4'-12 \\
		&= 2 n-|U_4'| - 1 \\
		&= 2 (n-|U_4'|) + |U_4'|-1 \\
		&\geq 2 (n-|U_4'|).
	\end{align*}
	Thus, $f_{k_3'}(n'-6) \geq n-|U_4'|$. Therefore, 
	$$f_{k_1'}(u_1')+f_{k_2'}(u_2')+f_{k_3'}(u_3') \geq f_{k_3'}(n'-6) \geq n - |U_4'|.$$
	
	Hence, there are at least $n-|U_4'|$ red copies of $K_2$ in $G[U_1'] \cup G[U_2'] \cup G[U_3']$. 
	
	Similarly, we pick $m_i \leq f_{k_i'}(u_i')$ red copies of $K_2$ in each $G[U_i']$ so that $m_1+m_2+m_3 = n-|U_4'|$.
	
	In addition, for $1\leq i \leq 3$, we have
	\begin{equation*}
		|U_i'| = 4k_i'+u_i' \geq 4k_i'+m_i \geq n+m_i = (m_1+m_2+m_3)+m_i+|U_4'|.
	\end{equation*}
	
	Denote $|U_4'|$ by $m_4$.
	There exists $m_iK_2 + \sum_{j\neq i} m_j K_1$ in $G[U_i']$ for $1 \leq i \leq 3$ and $m_4K_1$ in $G[U_4']$.
	Therefore, we get a red $nK_4$ in $G$.
	
	\textbf{Case 2:} $a_1,a_2,\ldots,a_t$ can be combined into two parts where each part is at least $n-2$.    
	
	We assume that $a_1,a_2,\ldots,a_t$ cannot be combined into three parts where each part is at least $\left \lceil{\frac{n}{4}}\right \rceil$.
	Similarly, we combine $V_1,V_2,\ldots,V_t$ and $V'$ into two sets, says $U_1$ and $U_2$, where $G[U_i]$ has $k_i \geq n-2$ disjoint blue copies of $K_4$ from $K$ and $u_i$ be the number of vertices in $U_i$ that do not belong to any of blue $K_4$ in $K$. 
	
	Note that, if $k \geq 2n-2+\left \lceil{\frac{n}{4}}\right \rceil$, then $a_3+a_4+\cdots+a_t \geq \left \lceil{\frac{n}{4}}\right \rceil$ as $a_1,a_2 \leq n-1$.
	Thus $a_1,a_2,\ldots,a_t$ can be combined into three parts where each part is at least $\left \lceil{\frac{n}{4}}\right \rceil$, which will be in the first case.
	
	We now consider $k \leq 2n-3+\left \lceil{\frac{n}{4}}\right \rceil$. 
	Therefore, $4k \leq 8n-12+4\left \lceil{\frac{n}{4}}\right \rceil$ and
	we have
	\begin{equation*}
		u_1+u_2 = 13n-3-4k \geq 13n-3-\bigg(8n-12+4\left \lceil{\frac{n}{4}}\right \rceil\bigg) \geq 4n+6.
	\end{equation*}
	We will assume that $u_1 \geq u_2$. Therefore, $u_1 \geq 2n+3$.
	
	\textbf{Case 2.1:} $u_1 \geq 2n+6$.
	By Theorem \ref{lem:mK_3}, we first recall that
	\begin{equation*}
		R\left(nK_3, (k_i+1)K_4\right) = \begin{cases} 3k_i+3n+4, & k_i< n-1, \\ 4k_i+2n+5, & k_i\geq n-1.\end{cases}
	\end{equation*}
	If $k_1 = n-2 < n-1$, then 
	\begin{equation*}
		|U_1| = 4k_1+u_1 \geq 4(n-2)+2n+6 = 3(n-2)+3n+4 = R(nK_3,(k_1+1)K_4).
	\end{equation*}
	Otherwise, if $k_1 \geq n-1$, then 
	\begin{equation*}
		|U_1| = 4k_1+u_1 > 4k_1+2n+5 = R(nK_3,(k_1+1)K_4).
	\end{equation*}
	Since $G[U_1]$ contains at most $k_1$ disjoint blue copies of $K_4$ from $K$, $G[U_1]$ contains a red $nK_3$.
	Since $|U_2| = 4k_2+u_2 > n$, $G[U_2]$ contains a red $nK_1$.
	Hence, we obtain a red $nK_4$ in $G$.
	
	\textbf{Case 2.2:} $2n+3 \leq u_1 \leq 2n+5$.
	Then $u_2 \geq 2n+1$. We then apply Lemma \ref{lem:fu} as follows.
	
	If $u_i \leq k_i+3$, we get
	\begin{equation*}
		f_{k_i}(u_i) = u_i-3 \geq 2n-2 \geq n.
	\end{equation*}
	If $u_i \geq k_i+4$ and $n \geq 3$, we get
	\begin{equation*}
		f_{k_i}(u_i) = \left \lfloor{\frac{u_i+k_i-2}{2}}\right \rfloor \geq \left \lfloor{\frac{2n+k_i-1}{2}}\right \rfloor \geq \left \lfloor{\frac{3n-3}{2}}\right \rfloor \geq n.
	\end{equation*}
	If $u_i \geq k_i+4$ and $n =2$, we get $k_i \geq 1 = n-1$. We then obtain
	\begin{equation*}
		f_{k_i}(u_i) = \left \lfloor{\frac{u_i+k_i-2}{2}}\right \rfloor \geq \left \lfloor{\frac{3n-2}{2}}\right \rfloor \geq n.
	\end{equation*}
	Thus, both $G[U_1]$ and $G[U_2]$ contains a red $nK_2$.
	Hence, we obtain a red $nK_4$ in $G$.
	
	In all cases, there exists a red $nK_4$ in $G$. This finishes the proof.
\end{proof}

\section{Conclusion and discussion}\label{sec:conclusion}
We have shown that $R_2(c(nK_4))=13n-3$ for $n \geq 3$ and the formula still holds for $n=2$ if $R_2(2K_4) \leq 23$.

The exact Ramsey number $R_2(c(nK_r))=(r^2-r+1)n-r+1$ when $r\geq 5$ are known when $n \geq R(r,r)$ \cite{rob}. 
However, this result may not always be true for small values of $n$. 
One example is when $n=2$ and $r=5$. If the result were true, we get $R_2(c(2K_5))=21(2)-4=38 < 43 \leq R(5,5)$, which is impossible.
Our proof relies on the known value $R(4,4)=18$.
However, for $r \geq 5$ the exact value of $R(r,r)$ remain unknown. 
Hence, it is interesting to know the upper bound of $R_2(c(nK_r))$, or in particular $R_2(c(nK_5))$, for small values of $n$.

\bibliographystyle{plain} 
\bibliography{refs}

\end{document}